\documentclass[aos,preprint]{imsart}

\usepackage{amsthm,amsmath,amssymb,algorithm,algorithmic,stmaryrd}
\usepackage{graphicx,ifthen}
\usepackage{tikz,calc} 
\usetikzlibrary{shapes,arrows,positioning}
\RequirePackage{hyperref}
\usepackage{natbib}

\RequirePackage[OT1]{fontenc}
\RequirePackage{amsthm,amsmath}
\RequirePackage[colorlinks,citecolor=blue,urlcolor=blue]{hyperref}
\usepackage[left=4cm,right=4cm,top=4cm,bottom=4cm]{geometry}
\makeatletter
\newtheorem*{rep@theorem}{\rep@title}
\newcommand{\newreptheorem}[2]{%
\newenvironment{rep#1}[1]{%
 \def\rep@title{#2 \ref{##1}}%
 \begin{rep@theorem}}%
 {\end{rep@theorem}}}
\makeatother

\newreptheorem{theorem}{Theorem}
\newreptheorem{lemma}{Lemma}
\newreptheorem{corollary}{Corollary}
\newreptheorem{example}{Example}
\newreptheorem{proposition}{Proposition}
\newtheorem{theorem}{Theorem}
\newtheorem{lemma}{Lemma}

\theoremstyle{remark}
\newtheorem{example}{Example}
\newtheorem{remark}{Remark}

\newcommand{\ignore}[1]{}

\newcommand{\bi}{\leftrightarrow}

\newcommand{\RDreg}{\mathbb{R}^{D}_{\text{reg}}}
\newcommand{\trace}{{\rm trace}}

\newcommand{\obs}[1]{{\rm mlt}_0(#1)}
\newcommand{\obsmu}[1]{{\rm mlt}(#1)}
\newcommand{\M}{\mathit{PD}}

\numberwithin{equation}{section}

\newcommand{\pa}{{\rm pa}}       
\newcommand{\Pa}{{\rm Pa}}       

\begin{document}

\begin{frontmatter}

\title{The Maximum Likelihood Threshold of a Path Diagram}
\runtitle{Maximum Likelihood Threshold of a Path Diagram}

\begin{aug}
  \author{\fnms{Mathias} \snm{Drton}\thanksref{m1}\ead[label=e1]{md5@uw.edu}},
  \author{\fnms{Christopher}
    \snm{Fox}\thanksref{m2}\ead[label=e2]{chrisfox@uchicago.edu}},
  \author{\fnms{Andreas} \snm{K\"{a}ufl}\thanksref{m3} \ead[label=e3]{andreas.kaeufl@googlemail.com}}
  \and
  \author{\fnms{Guillaume} \snm{Pouliot}\thanksref{m4} \ead[label=e4]{guillaume.pouliot@gmail.com}}

\address{Department of Statistics\\
University of Washington\\
Seattle,   WA, U.S.A.\\
\printead{e1}\\}
\address{Department of Statistics\\
The University of Chicago\\
Chicago, IL, U.S.A.\\
\printead{e2}\\
\phantom{E-mail: {\rm guillaume.pouliot@gmail.com}}
}
\address{Institute for Mathematics\\
University of Augsburg\\
Augsburg, Germany\\
\printead{e3}\\}
\address{Harris School of Public Policy\\
The University of Chicago\\
Chicago, IL, U.S.A.\\[0.1cm]
Department of Mathematics\\
 and Industrial Engineering\\
\'Ecole Polytechnique de Montr\'eal\\
Montr\'eal, QC, Canada\\
\printead{e4}\\}

\runauthor{M. Drton et al.}

\affiliation{University of Washington\thanksmark{m1}, University of Chicago\thanksmark{m2}\thanksmark{m4},
  University of Augsburg\thanksmark{m3} and \'Ecole Polytechnique de Montr\'eal\thanksmark{m4}}

\end{aug}

\begin{abstract}
Linear structural equation models postulate noisy linear relationships
between variables of interest.  Each model corresponds to a path
diagram, which is a mixed graph with directed edges that encode the
domains of the linear functions and bidirected edges that indicate
possible correlations among noise terms.  Using this graphical
representation, we determine the maximum likelihood threshold, that
is, the minimum sample size at which the likelihood function of a
Gaussian structural equation model is almost surely bounded.  Our
result allows the model to have feedback loops and is based on
decomposing the path diagram with respect to the connected components
of its bidirected part.  We also prove that if the sample size is
below the threshold, then the likelihood function is almost surely
unbounded.  Our work clarifies, in particular, that standard
likelihood inference is applicable to sparse high-dimensional models
even if they feature feedback loops.
\end{abstract}


\end{frontmatter}

	
\section{Introduction}\label{sec:intro}

Structural equation models are multivariate statistical models that
treat each variable of interest as a function of the remaining
variables and a random error term. Linear structural equation models require
all these functions to be linear.  Let $X = (X_1,\ldots,X_p)$ be the random
vector holding the considered variables.  Then $X$ solves the equation system
\begin{equation}
	X_i = \lambda_{0i}+\sum_{j \neq i} \lambda_{ij} X_j + \epsilon_i, \quad i = 1,\ldots,p,
\label{eq:structuraleqs}
\end{equation}
where $\epsilon = (\epsilon_1,\ldots,\epsilon_p)$ is a given $p$-dimensional
random error vector, and the $\lambda_{0i}$ and $\lambda_{ij}$ are
unknown parameters.  Let $\Lambda_0=(\lambda_{01},\dots,\lambda_{0p})$
and form the matrix
$\Lambda = (\lambda_{ij})\in\mathbb{R}^{p\times p}$ by setting the
diagonal entries to zero.  Following the frequently made Gaussian
assumption, assume that $\epsilon$ is centered $p$-variate normal with
covariance matrix $\Omega = (\omega_{ij})$.  Writing $I$ for
the 
identity matrix,~(\ref{eq:structuraleqs}) yields that
$X=(I-\Lambda)^{-T}\epsilon$ is multivariate normal with covariance
matrix
\begin{equation}
	\Sigma = (I-\Lambda)^{-T} \Omega (I-\Lambda)^{-1}.
\label{eq:covariancematrix}
\end{equation}
Here, and throughout the paper, the matrix $I-\Lambda$ is required to be
invertible.

The Gaussian models just introduced have a long tradition
\citep{wright:1921,wright:1934} but remain an important tool for
modern applications
\citep[e.g.,][]{maathuis-colombo-kalisch-2009,nature}.  Their
popularity is driven by causal interpretability
\citep{pearl:2009,spirtes:2000} as well as favorable statistical
properties that facilitate analysis of highly multivariate data.  In
this paper, we focus on the fact that if the matrices $\Lambda$ and
$\Omega$ are suitably sparse, then maximum likelihood estimates in
high-dimensional models may exist at small sample sizes.  This
enables, for instance, the use of likelihood in stepwise model selection.
It can often be expected that $\Lambda$ is sparse because each
variable $X_i$ depends on only a few of the other variables $X_j$,
$j\not=i$.  Similarly, the number of nonzero off-diagonal entries of
$\Omega$ is small unless many pairs of error terms $\epsilon_i$ and
$\epsilon_j$ are correlated through a latent common cause for $X_i$
and $X_j$.
We encode assumptions of sparsity in $(\Lambda,\Omega)$ in a graphical
framework advocated by \cite{wright:1921,wright:1934}.  Our
terminology follows conventions from the book of
\cite{lauritzen:1996}, the review of
\cite{hibi}, and other related work such as
\cite{foygel:draisma:drton:2012} or \cite{evans:2016}.

\subsection*{Background}

A \emph{mixed graph} is a triple $\mathcal{G}=(V,D,B)$ such that
$D\subseteq V\times V$ and $B$ is a set containing 2-element subsets
of $V$.    Throughout the paper, we take the
vertex set to be $V = \{1,\ldots,p\}$ such that the nodes in
$V$ index the given random variables $X_1,\dots,X_p$.  The
pairs $(i,j)\in D$ are \emph{directed edges} that we denote as
$i \to j$.   Node $j$ is the head of such an edge.  We always
assume that there are no self-loops, that is, $i \to i \notin D$ for
all $i\in V$.  The elements $\{i,j\}\in B$ are \emph{bidirected edges}
that have no orientation; we write such an edge as $i \bi j$ or
$j\bi i$.  Two nodes $i,j\in V$ are \emph{adjacent} if $i\bi j\in B$
or $i\to j\in D$ or $j\to i\in D$.

Let $\mathcal{G}'=(V',D',B')$ be another mixed graph.  If $V'\subseteq V$,
$D'\subseteq D$, and $B'\subseteq B$, then $\mathcal{G}'$ is a \emph{subgraph}
of $\mathcal{G}$, and $\mathcal{G}$ \emph{contains} $\mathcal{G}'$.  If
$V'=\{i_0,i_1,\dots,i_k\}$ for distinct $i_0,i_1,\dots,i_k$ and there
are $|D'|+|B'|=k$ edges such that any two consecutive nodes $i_{h-1}$
and $i_h$ are adjacent in $\mathcal{G}'$, then $\mathcal{G}'$ is a \emph{path} from $i_0$
to $i_k$.  It is a \emph{directed path} if $i_{h-1}\to i_h$ for all
$h$.  Adding the edge $i_k\to i_0$ gives a \emph{directed cycle}.

A mixed graph $\mathcal{G}$ is \emph{connected} if it contains a path
from any node $i$ to any other node $j$.  A \emph{connected component}
of $\mathcal{G}$ is an inclusion-maximal connected subgraph.  In other
words, a subgraph $\mathcal{G}'$ is a connected component of
$\mathcal{G}$ if $\mathcal{G}'$ is connected and every subgraph of
$\mathcal{G}$ that strictly contains $\mathcal{G}'$ fails to be
connected.  If $\mathcal{G}$ does not contain any directed cycles,
then it is \emph{acyclic}.  If it has only directed edges
($B=\emptyset$), then $\mathcal{G}$ is a \emph{digraph}.  The
graphical modeling literature refers to an \emph{acyclic digraph} also
as directed acyclic graph, abbreviated to DAG.

Now, let $\mathbb{R}^D$ be the space of real $p \times p$ matrices
$\Lambda = (\lambda_{ij})$ with $\lambda_{ij}=0$ whenever
$i \to j \notin D$, and write $\RDreg$ for the subset of matrices
$\Lambda\in\mathbb{R}^D$ with $I-\Lambda$ invertible.  Note that
$\mathbb{R}^D = \RDreg$ if and only if $\mathcal{G}$ is {acyclic}.
Let $\M(B)$ be the cone of positive definite
$p \times p$ matrices $\Omega=(\omega_{ij})$ with $\omega_{ij} = 0$
when $i \neq j$ and $i \bi j \notin B$.  Then the linear structural
equation model given by $\mathcal{G}$ is the set of multivariate normal
distributions $\mathcal{N}(\mu,\Sigma)$ with mean vector
$\mu\in\mathbb{R}^p$ and covariance matrix $\Sigma$ in
\begin{equation}
\M(\mathcal{G}) \;=\; \{ (I-\Lambda)^{-T}\Omega(I-\Lambda)^{-1}:
(\Lambda,\Omega)\in \RDreg\times \M(B)\}.
\label{eq:model}
\end{equation}
We remark that the graph $\mathcal{G}$ is also known as the \emph{path diagram}
of the model.

\subsection*{Maximum likelihood threshold}

Suppose now that we have independent and identically distributed
multivariate observations
$X^{(1)},\ldots,X^{(n)}\sim\mathcal{N}(\mu,\Sigma)$.  Let
\begin{equation}
  \label{eq:sample-mean-cov}
\bar X_n \;=\; \frac{1}{n}\sum_{s=1}^n X^{(s)} \quad\text{and}\quad S_n
\;=\; \frac{1}{n}\sum_{s=1}^n(X^{(s)}-\bar X_n)(X^{(s)}-\bar X_n)^T
\end{equation}
be the sample mean vector and sample covariance matrix.  With
an additive constant omitted and $n/2$ divided out, the log-likelihood
function is
\begin{equation*}
  \ell(\mu,\Sigma\,|\, \bar X_n,S_n) = - \log\det(\Sigma) -
  \trace\left(\Sigma^{-1} S_n\right) - (\bar
  X_n-\mu)^T\Sigma^{-1}(\bar X_n-\mu).
\end{equation*}
The considered models have the mean vector unrestricted and the
maximum likelihood estimator of $\mu$ is always $\bar X_n$.  This
yields the profile log-likelihood function
\begin{equation}\label{eq:likelihood}
  \ell(\Sigma\,|\, S_n) = - \log\det(\Sigma) -
  \trace\left(\Sigma^{-1} S_n\right).
\end{equation}

Our interest is in determining, for a mixed graph $\mathcal{G} = (V,D,B)$, the
minimum number $N$ such that for a sample of size $n\ge N$ 
the log-likelihood function is almost surely bounded above on the set
$\mathbb{R}^p \times \M(\mathcal{G})$.  As usual, almost surely refers to
probability one when $X^{(1)},\ldots,X^{(n)}$ are an independent
sample from a regular multivariate normal distribution, or
equivalently, any other absolutely continuous distribution on
$\mathbb{R}^p$.  Let
\[
  \hat\ell(\mathcal{G}\,|\,S_n) \;=\; \sup \left\{ \ell(\Sigma\,|\,S_n) :
    \Sigma\in\M(\mathcal{G})\right\}. 
\]
Adapting terminology from \cite{Sullivant:Gross}, the number we seek
to derive is the \emph{maximum likelihood threshold}
\begin{equation}
  \label{eq:obsG}
  \obsmu{\mathcal{G}} \;:=\; \min \left\{N\in\mathbb{N} \,:\,  
    \,\hat\ell(\mathcal{G}\,|\,
    S_n)  <\infty \;\; \text{a.s.} \;\;   \forall n\ge N
\right\}.
\end{equation}
Here and throughout, a.s.~abbreviates almost surely.  

If we constrain the mean vector $\mu$ to be zero, then the relevant
sample covariance matrix is
\begin{equation}
  S_{0,n}=\frac{1}{n} \sum_{s=1}^n X^{(s)}(X^{(s)})^T.
  \label{eq:Sn0}
\end{equation}
By classical results \citep[Chap.~7]{MR1990662}, $\obsmu{\mathcal{G}}=\obs{\mathcal{G}}+1$, where
\begin{equation}
  \label{eq:obs0G}
  \obs{\mathcal{G}} \;=\; \min \left\{N\in\mathbb{N} \,:\,  \hat\ell(\mathcal{G}\,|\, S_{0,n}) <\infty \;\;\text{a.s.} \;\;   \forall n\ge N\right\}
\end{equation}
is the maximum likelihood threshold for the model when the mean vector
is taken to be zero.
Our subsequent discussion will thus focus on the threshold $\obs{\mathcal{G}}$.
We record three simple yet useful facts.  

\begin{lemma}
  \label{lem:saturated}
  Let $\mathcal{G}=(V,D,B)$ be a mixed graph.  Then
  \begin{enumerate}
  \item[(a)] $\obs{\mathcal{G}} \;\le\; p=|V|$.
  \item[(b)] If $\mathcal{G}_1,\dots,\mathcal{G}_k$ are the connected components of $\mathcal{G}$, then
    \[
      \obs{\mathcal{G}} \;=\; \max_{j=1,\dots,k} \,\obs{\mathcal{G}_j}.
    \]
  \item[(c)] If $\mathcal{H}$ is a subgraph of $\mathcal{G}$, then $\obs{\mathcal{H}}\le\obs{\mathcal{G}}$.
  \end{enumerate}
\end{lemma}
\begin{proof}
  (a) It is well known that $\ell(\cdot\,|\,S_{0,n})$ is bounded above
  on the entire cone of positive definite matrices if and only if
  $S_{0,n}$ is positive definite.  Moreover, if $S_n$ is positive,
  then $\Sigma=S_n$ is the unique maximizer \cite[Lemma
  3.2.2]{MR1990662}.  The matrix $S_{0,n}$ is positive definite
  a.s. if and only if $n\ge p$.

  (b) The variables in the different connected components are
  independent.  The likelihood function may be maximized separately
  for the different components.

  (c) If $\mathcal{H}$ and $\mathcal{G}$ have the same vertex set, then
  $\M(\mathcal{H})\subseteq \M(\mathcal{G})$ and, thus,
  $\hat\ell(\mathcal{H}\,|\, S_{0,n})\le \hat\ell(\mathcal{G}\,|\,S_{0,n})$.  The case
  where $\mathcal{H}$ has fewer vertices can be addressed by adding
  isolated nodes and using the fact from (b).
\end{proof}

When $\mathcal{G}$ is connected, Lemma~\ref{lem:saturated} yields only
the trivial bound $\obs{\mathcal{G}} \le p$.  However,
$\obs{\mathcal{G}}$ may be far smaller than $p$ when $\mathcal{G}$ is
sparse, that is, has few edges.  Indeed, in the well understood case
of $\mathcal{G}$ being an acyclic digraph, maximum likelihood
estimation reduces to solving one linear regression problem for each
considered variable \citep[p.~154]{lauritzen:1996}.  The predictors in
the problem for variable $j$ are the variables from the set of {\em
  parents\/} $\pa(j)=\{k\in V:k\to j\in D\}$.  If the sample size
exceeds the size of the largest parent set, then at least one degree
of freedom remains for estimation of the error variance in each one of
the $p$ linear regression problems.  We thus have the following
well-known fact.
%
\begin{theorem}\label{thm:dags}
  Let $\mathcal{G}=(V,D,\emptyset)$ be an acyclic digraph. Then
  \begin{equation*}
    \obs{\mathcal{G}} \;=\; 1+ \max_{j\in V} |\pa(j)|.
  \end{equation*}
\end{theorem}
The quantity $|\pa(j)|$ in the theorem is also termed the in-degree of
node $j$.
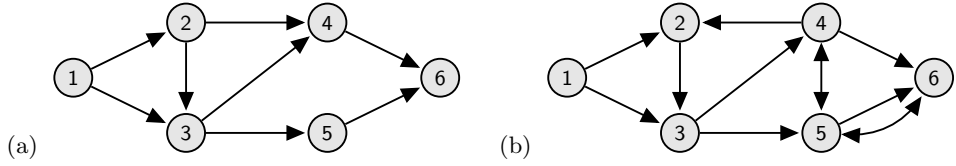
\begin{figure}[t]
  \centering
   \begin{center}
     \begin{tabular}{ll}
     (a)
    \scalebox{.9}{
      \begin{tikzpicture}[->,>=triangle 45,shorten >=1pt,
        auto,thick,
        main node/.style={circle,inner
          sep=3pt,fill=gray!20,draw,font=\sffamily}]  
        
        \node[main node] (1) {1};
        \node[main node] (2) [above right=.4cm and 1.25cm of 1]{2}; 
        \node[main node] (3) [below right=.4cm and 1.25cm of 1] {3}; 
        \node[main node] (4) [right=1.5cm of 2] {4};
        \node[main node] (5) [right=1.5cm of 3] {5};
        \node[main node] (6) [above right=.4cm and 1.25cm of 5] {6};
      
        \path[every node/.style={font=\sffamily\small}] 
        (1) edge node {} (2)
        (1) edge node {} (3)
        (2) edge node {} (3)
        (2) edge node {} (4)
        (3) edge node {} (4)
        (3) edge node {} (5)
        (4) edge node {} (6)
        (5) edge node {} (6);
      \end{tikzpicture}
    }
       & (b)
         \scalebox{.9}{
      \begin{tikzpicture}[->,>=triangle 45,shorten >=1pt,
        auto,thick,
        main node/.style={circle,inner
          sep=3pt,fill=gray!20,draw,font=\sffamily}]  
        
        \node[main node] (1) {1};
        \node[main node] (2) [above right=.4cm and 1.25cm of 1]{2}; 
        \node[main node] (3) [below right=.4cm and 1.25cm of 1] {3}; 
        \node[main node] (4) [right=1.5cm of 2] {4};
        \node[main node] (5) [right=1.5cm of 3] {5};
        \node[main node] (6) [above right=.4cm and 1.25cm of 5] {6};
      
        \path[every node/.style={font=\sffamily\small}] 
        (1) edge node {} (2)
        (1) edge node {} (3)
        (2) edge node {} (3)
        (4) edge node {} (2)
        (3) edge node {} (4)
        (3) edge node {} (5)
        (4) edge node {} (6)
        (5) edge node {} (6);
        \path[<->,every node/.style={font=\sffamily\small}, bend right] 
        (5) edge node {} (6);
        \path[<->,every node/.style={font=\sffamily\small}] 
        (4) edge node {} (5)
        ;        
      \end{tikzpicture}
    }
      \end{tabular}
  \end{center}
  \caption{(a) An acyclic digraph with $\obs{\mathcal{G}}=3$.  (b) A mixed graph
    with $\obs{\mathcal{G}}=4$.}
  \label{fig:intro_example}
\end{figure}

\begin{example}
  If $\mathcal{G}$ is the acyclic digraph from Figure
  \ref{fig:intro_example}(a), then the largest parent sets are of size
  two, for nodes $j\in\{3,4,6\}$. By Theorem \ref{thm:dags},
  $\obs{\mathcal{G}}=3$.
\end{example}

\subsection*{Main result}

In this paper, we determine $\obs{\mathcal{G}}$ for any mixed graph
$\mathcal{G}=(V,D,B)$.  For a set $A\subseteq V$, let $\Pa(A)$ be the
union of $A$ and the parents of its elements, so
\begin{equation}
  \label{eq:PaA}
  \Pa(A) \;=\; A \cup \bigcup_{i\in A} \pa(i).
\end{equation}
Then our main result can be stated as follows.
\begin{theorem}\label{thm:intro}
  Let $\mathcal{G} = (V,D,B)$ be a mixed graph, and let $C_1,\dots,C_l$ be the
  vertex sets of the connected components of its bidirected part
  $\mathcal{G}_\bi=(V,\emptyset,B)$. Then
  \[
    \obs{\mathcal{G}} \;=\; \max_{j=1,\dots,l} \,|\Pa(C_j)|.
  \]
  Moreover, if $n<\obs{\mathcal{G}}$ then $\hat\ell(\mathcal{G}\,|\, S_{0,n})=\infty$ a.s.
\end{theorem}
In the special case that $\mathcal{G}$ is an acyclic digraph, we have
$B=\emptyset$ and Theorem~\ref{thm:intro} reduces to
Theorem~\ref{thm:dags} because each connected component of $\mathcal{G}_\bi$
has only a single node $j\in V$.  Then $\Pa(\{j\})=\pa(j)\cup\{j\}$
and $|\Pa(\{j\})|=1+|\pa(j)|$.

\begin{example}
  Let $\mathcal{G}$ be the graph in Figure \ref{fig:intro_example}(b).
  The parameters of the model given by $\mathcal{G}$ are identifiable
  in a generic or almost everywhere sense, as can be checked readily
  using the half-trek criterion
  \citep{foygel:draisma:drton:2012,semid}.  Hence, $\M(\mathcal{G})$
  is a 16-dimensional subset of the 21-dimensional cone of positive
  definite $6\times 6$ matrices.
  By Theorem \ref{thm:intro}, $\obs{\mathcal{G}} = 4$.  Indeed, $\mathcal{G}_\bi$ has four
  connected components with vertex sets $C_1=\{1\}$, $C_2=\{2\}$,
  $C_3=\{3\}$ and $C_4=\{4,5,6\}$.  Adding parents yields
  $\Pa(C_1)=\{1\}$, $\Pa(C_2)=\{1,2,4\}$, $\Pa(C_3)=\{1,2,3\}$, and
  $\Pa(C_4)=\{3,4,5,6\}$.
\end{example}

\begin{remark}
  \label{rem:not-closed}
  If the likelihood function associated with an acyclic digraph
  $\mathcal{G}=(V,D,\emptyset)$ is bounded then it achieves its maximum.  Hence,
  $n\ge \obs{\mathcal{G}}$ ensures that the maximum is a.s.~achieved.  We are
  not aware of any results in the literature that, for a more general
  class of graphs, would similarly guarantee achievement of the
  maximum.  In fact, we believe that there are mixed graphs $\mathcal{G}$ such
  that even for sample size $n\ge\obs{\mathcal{G}}$ the probability of the
  likelihood function failing to achieve its maximum is not zero.
  This belief is based on the fact that the set $\M(\mathcal{G})$ is not
  generally closed.  As a simple example, consider the graph $\mathcal{G}$ with
  edges $1\to 2$, $2\to 3$, and $2\bi 3$, for which $\M(\mathcal{G})$ comprises
  all positive definite $3\times 3$ matrices $\Sigma=(\sigma_{ij})$
  with $\sigma_{13}=0$ whenever $\sigma_{12}=0$.
\end{remark}

\subsection*{Outline}

In the remainder of the paper, we first prove that $\obs{\mathcal{G}}$
is no larger than the value asserted in Theorem~\ref{thm:intro}
(Section~\ref{sec:upper-bound}).  Next, we derive $\obs{\mathcal{G}}$
for any bidirected graph $\mathcal{G}$ (Section~\ref{sec:bidirected}).
In Section~\ref{sec:lower-bound} we use submodels given by bidirected
graphs to show that the value from Theorem~\ref{thm:intro} is also a
lower bound on $\obs{\mathcal{G}}$ for any (possibly cyclic) mixed
graph, which then completes the proof of Theorem~\ref{thm:intro}.  A
numerical experiment in Section~\ref{sec:numerical-experiment}
exemplifies that even a high-dimensional model is amenable to standard
likelihood inference as long as its maximum likelihood threshold is
small.  The experiment suggests that likelihood inference allows one
to perform model selection for high-dimensional but sparse cyclic
models.  In Section~\ref{sec:undirected}, we highlight interesting
differences between the maximum likelihood threshold of Gaussian
graphical models given by a directed versus an undirected cycle.  The
former model is nested in the latter and the two models have the same
dimension, yet the thresholds are different.

\section{Upper bound on the sample size
  threshold}\label{sec:upper-bound}

We prove the upper bound that is part of Theorem~\ref{thm:intro}.

\begin{theorem}
  \label{thm:bounding-above-cyclic}
  Let $\mathcal{G} = (V,D,B)$ be a mixed graph, and let $C_1,\dots,C_l$ be the
  vertex sets of the connected components of its bidirected part
  $\mathcal{G}_\bi=(V,\emptyset,B)$. Then
  \[
    \obs{\mathcal{G}} \;\le\; \max_{j=1,\dots,l} \,|\Pa(C_j)|.
  \]
\end{theorem}
\begin{proof}
  Let $\mathcal{G}'$ be the supergraph of $\mathcal{G}$ obtained by adding bidirected
  edges between any two nodes that are in the same connected component
  of $\mathcal{G}_{\bi}=(V,\emptyset,B)$ but that are not adjacent in $\mathcal{G}_{\bi}$.  Then
  $C_1,\dots,C_l$ are still the vertex sets of the connected
  components of the bidirected part of $\mathcal{G}'$, and the sets $\Pa(C_j)$
  are identical in $\mathcal{G}$ and $\mathcal{G}'$; see Figure~\ref{fig:intro_example:G'}
  for an example.  We emphasize that the bidirected part of $\mathcal{G}'$ is a
  disjoint union of complete subgraphs.  The remainder of this proof
  shows the claimed bound for $\mathcal{G}'$.  By Lemma~\ref{lem:saturated}(c),
  the bound then also holds for $\mathcal{G}$.  To simplify notation, we assume
  that $\mathcal{G}$ itself has a bidirected part $\mathcal{G}_{\bi}$ that is a disjoint
  union of complete graphs.

\begin{figure}[t]
  \centering
   \begin{center}
    \scalebox{.9}{
      \begin{tikzpicture}[->,>=triangle 45,shorten >=1pt,
        auto,thick,
        main node/.style={circle,inner
          sep=3pt,fill=gray!20,draw,font=\sffamily}]  
        
        \node[main node] (1) {1};
        \node[main node] (2) [above right=.4cm and 1.25cm of 1]{2}; 
        \node[main node] (3) [below right=.4cm and 1.25cm of 1] {3}; 
        \node[main node] (4) [right=1.5cm of 2] {4};
        \node[main node] (5) [right=1.5cm of 3] {5};
        \node[main node] (6) [above right=.4cm and 1.25cm of 5] {6};
      
        \path[every node/.style={font=\sffamily\small}] 
        (1) edge node {} (2)
        (1) edge node {} (3)
        (2) edge node {} (3)
        (4) edge node {} (2)
        (3) edge node {} (4)
        (3) edge node {} (5)
        (4) edge node {} (6)
        (5) edge node {} (6);
        \path[<->,every node/.style={font=\sffamily\small}, bend right] 
        (5) edge node {} (6)
        (6) edge node {} (4);
        \path[<->,every node/.style={font=\sffamily\small}] 
        (4) edge node {} (5)
        ;        
      \end{tikzpicture}
    }
  \end{center}
  \caption{The graph $\mathcal{G}'$ when $\mathcal{G}$ is the mixed graph from
    Figure~\ref{fig:intro_example}(b).  Edge $4\bi 6$ has been added.} 
  \label{fig:intro_example:G'}
\end{figure}
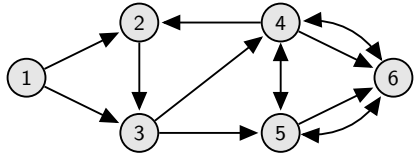

  For $\Sigma=(I-\Lambda)^{-T}\Omega(I-\Lambda)^{-1}$, we have
  \begin{align*}
    \ell(\Sigma\,|\,S_{0,n}) \;=\;
    \log \left( \det(I-\Lambda)^2\right) - \log\det(\Omega) -
    \trace\left((I-\Lambda)\Omega^{-1}(I-\Lambda)^T S_{0,n}\right).
  \end{align*}
  The set $\M(B)$ comprises all block-diagonal $p\times p$ matrices
  with $l$ blocks determined by the connected components
  of $\mathcal{G}_{\bi}$.  Therefore, if $\Lambda=(\lambda_{jk})\in\RDreg$ and
  $\Omega\in\M(B)$, we have
  \begin{multline}
    \label{eq:ell-digraph}
    \ell(\Sigma\,|\,S_{0,n}) 
    \;=\;
      \log \left( \det(I-\Lambda)^2\right) \\
      - \sum_{j=1}^l \left[ \log\det\left(
      \Omega_{C_j,C_j}\right) + \trace\left\{\Omega_{C_j,C_j}^{-1}
      \left((I-\Lambda)^T 
      S_{0,n} (I-\Lambda)\right)_{C_j,C_j}\right\}\right]. 
  \end{multline}
  Let $X_1,\dots,X_p$ be the columns of the data matrix
  \[
    \mathbf{X} \;=\;
    \begin{pmatrix}
      X^{(1)} & \dots & X^{(n)}
    \end{pmatrix}^T
    \in\mathbb{R}^{n\times p}.
  \]
  Then $S_{0,n}=\frac{1}{n} \mathbf{X}^T\mathbf{X}$, and
  \begin{align*}
    \label{eq:sos}
    \hat\Omega_{C_j,C_j} 
    &
      \;=\; \left[(I-\Lambda)^T
      S_{0,n} (I-\Lambda)\right]_{C_j,C_j}\\
    &\;=\; 
      \frac{1}{n} \left[ \mathbf{X}(I-\Lambda)_{V,C_j} \right]^T\left[
        \mathbf{X}(I-\Lambda)_{V,C_j} \right]
  \end{align*}
  is the sample covariance matrix of the vector of error terms 
  \[
    X_u-\sum_{k\in\pa(u)} \lambda_{ku}X_k, \quad u\in C_j.
  \]

  Fix $\Lambda\in\RDreg$.  Then, for any $j=1,\dots,l$, the function
  \begin{equation}
    \label{eq:ell-j}
    \Omega_{C_j,C_j} \mapsto -\log\det\left(
      \Omega_{C_j,C_j}\right) - \trace\left\{\Omega_{C_j,C_j}^{-1}
      \left[(I-\Lambda)^T 
      S_{0,n} (I-\Lambda)\right]_{C_j,C_j}\right\}
  \end{equation}
  is bounded if and only if $\hat\Omega_{C_j,C_j}$ is positive
  definite.  If it is bounded then $\hat\Omega_{C_j,C_j}$ is the
  unique maximizer \cite[Lemma 3.2.2]{MR1990662}.  We claim that if
  $n\ge |\Pa(C_j)|$, then $\hat\Omega_{C_j,C_j}$ is a.s.~positive
  definite.  Indeed, by the lemma in \cite{okamoto:1973}, all square
  submatrices of $\mathbf{X}$ are a.s.~invertible.  If
  $n\ge |\Pa(C_j)|$, this implies that the vectors $X_k$,
  $k\in \Pa(C_j)$, are a.s.~linearly independent.  The columns of
  $\mathbf{X}(I-\Lambda)_{V,C_j}$ are linear combinations of these vectors.
  Because $I-\Lambda$ is invertible, the submatrix
  $(I-\Lambda)_{V,C_j}$ has full column rank $|C_j|$.  Therefore,
  $\mathbf{X}(I-\Lambda)_{V,C_j}$ a.s.~has full column rank $|C_j|$, which
  implies positive definiteness of $\hat\Omega_{C_j,C_j}$.

  Because a union of null sets is a null set, if
  $n\ge \max_{j=1,\dots,l} |\Pa(C_j)|$, then a.s.~all matrices
  $\hat\Omega_{C_j,C_j}$ for $j=1,\dots, l$ are simultaneously
  positive definite.  We may thus proceed by substituting all
  $\hat\Omega_{C_j,C_j}$ into the log-likelihood function
  $\ell(\Sigma\,|\,S_{0,n})$ displayed in~(\ref{eq:ell-digraph}).  The
  resulting 
  profile log-likelihood function is
  \begin{equation}
    \label{eq:profile}
    \ell(\Lambda\,|\,S_{0,n}) \;=\; 
    \log \left( \det(I-\Lambda)^2\right) -\, p\, - \sum_{j=1}^l
    \log\det\left(\left[(I-\Lambda)^T 
      S_{0,n} (I-\Lambda)\right]_{C_j,C_j}\right). 
\end{equation}
  In order to show that $\ell(\Lambda\,|\,S_{0,n})$ is a.s.~bounded from
  above, we apply a block-version of the Hadamard inequality, which
  yields that
  \begin{equation}
    \label{eq:hadamard}
    \log \left( \det(I-\Lambda)^2\right) \;\le\; \sum_{j=1}^l
    \log\det\left( \left[(I-\Lambda)^T 
      (I-\Lambda)\right]_{C_j,C_j}\right);
  \end{equation}
  recall that the sets $C_j$ form a partition of $V=\{1,\dots,p\}$.
  Using~(\ref{eq:hadamard}) in~(\ref{eq:profile}), we see that up to a
  constant the exponential of $\ell(\Lambda\,|\,S_{0,n})$ is bounded
  above by the product of the terms
  \begin{multline}
    \label{eq:profile-j}
    \frac{\det\left(\left[(I-\Lambda)^T 
          (I-\Lambda)\right]_{C_j,C_j}\right)}{\det\left(\left[(I-\Lambda)^T 
          S_{0,n} (I-\Lambda)\right]_{C_j,C_j}\right)}\\
    \;=\;
    \frac{\det\left((I-\Lambda^T)_{C_j,\Pa(C_j)} 
        (I-\Lambda)_{\Pa(C_j),C_j}\right)}{\det\left((I-\Lambda^T)_{C_j,\Pa(C_j)} 
        (S_{0,n})_{\Pa(C_j),\Pa(C_j)} (I-\Lambda)_{\Pa(C_j),C_j}\right)}
  \end{multline}
  for $j=1,\dots,l$.
  Let $\lambda_j(S_{0,n})$ be the minimum eigenvalue of the
  $\Pa(C_j)\times \Pa(C_j)$ submatrix of $S_{0,n}$.  This submatrix is
  the sample covariance matrix of the variables indexed by $\Pa(C_j)$.
  Therefore, if $n\ge |\Pa(C_j)|$, then $\lambda_j(S_{0,n})$ is
  a.s.~positive.  Now,
  $(S_{0,n})_{\Pa(C_j),\Pa(C_j)}\succeq \lambda_j(S_{0,n})I$ in the
  positive semidefinite ordering.  Using Observation 7.2.2 and
  Corollary 7.7.4(b) in \cite{MR1084815}, we obtain that the ratio
  in~(\ref{eq:profile-j}) is a.s.~bounded above by
  $\lambda_j(S_{0,n})^{-|C_j|}<\infty$.
\end{proof}

\section{Bidirected graphs}\label{sec:bidirected}

Consider a bidirected graph $\mathcal{G}=(V,\emptyset,B)$.  Then $\M(\mathcal{G}) = \M(B)$
is a set of sparse positive definite matrices.  We
prove that the bound from Lemma~\ref{lem:saturated}(a) is an equality
when the bidirected graph $\mathcal{G}$ is connected.
\begin{theorem}\label{thm:bidirected}
  If $\mathcal{G}=(V,\emptyset,B)$ is connected, then $\obs{\mathcal{G}} = p$.  Moreover,
  if $n<\obs{\mathcal{G}}$ then $\hat\ell(\mathcal{G}\,|\, S_{0,n})=\infty$ a.s.
\end{theorem}
The proof of the theorem makes use of two lemmas.  We derive those first.

\begin{lemma}\label{lem:zero_entries}
  If $n<p$, then the kernel of $S_{0,n}$ a.s.~contains a vector
  $q\in\mathbb{R}^p$ with all coordinates nonzero.
\end{lemma}
\begin{proof}
  The matrix $S_{0,n}$ has the same kernel as 
  \[
    \mathbf{X} \;=\;
    \begin{pmatrix}
      X^{(1)} & \dots & X^{(n)}
    \end{pmatrix}^T
    \in\mathbb{R}^{n\times p}.
  \]
  Partition the matrix as $\mathbf{X}=(\mathbf{X}_1,\mathbf{X}_2)$, where the square submatrix $\mathbf{X}_1$
  contains the first $n$ columns.  
  The determinant being a polynomial, the lemma in \cite{okamoto:1973}
  yields that $\mathbf{X}_1$ is a.s.~invertible.

  We claim that for all $j\le n$, the kernel of $\mathbf{X}$ almost
  surely contains a vector $q$ with $q_{n+1}=\dots=q_p=1$ and
  $q_j\not=0$.  Without loss of generality, it suffices to treat the
  case of $j=1$.  By the above discussion, we may assume that
  $\mathbf{X}_1$ is invertible.  Then a partitioned vector
  $(u,v)\in\mathbb{R}^p$ is in the kernel of $\mathbf{X}$ if and only
  if $u=-\mathbf{X}_1^{-1}\mathbf{X}_2v$.  Let
  $e=(1,\dots,1)^T\in\mathbb{R}^{p-n}$.  The claim is true if and only
  if the vector $u=-\mathbf{X}_1^{-1}\mathbf{X}_2e$ has first entry
  $u_1\not=0$.  Multiplying $u_1$ with $\det(\mathbf{X}_1)$ gives a
  polynomial $f(\mathbf{X})$ such that $u_1=0$ only if
  $f(\mathbf{X})=0$.  The lemma in \cite{okamoto:1973} yields the
  claim if we can argue that the product
  $f(\mathbf{X})\det(\mathbf{X}_1)$ is not the zero polynomial.  To
  this end, it is enough to exhibit one matrix $\mathbf{X}$ such that
  $u_1\not=0$ and $\det(\mathbf{X}_1)\not=0$.  Take $\mathbf{X}_1=I$
  and let $\mathbf{X}_2$ have a single nonzero entry
  $\mathbf{X}_{1,n+1}=-1$.  Then $u=(1,0,\dots,0)^T$.

  Because a union of null sets is a null set, the kernel of $\mathbf{X}$ almost
  surely contains a vector $q$ with $q_{n+1}=\dots=q_p=1$ and
  $q_j\not=0$ for all $j\le n$.  
\end{proof}

\begin{lemma}\label{lem:constr_barsigma}
  Let $q$ be any vector with all entries nonzero.  There exists
  a matrix $\Sigma \in \M(\mathcal{G})$, such that the vector $\Sigma q$ has
  precisely one nonzero entry.
\end{lemma}
\begin{proof}
  For a subset of nodes $A \subset V$, let $\mathcal{G}_A = (A,\emptyset,B_A)$
  be the subgraph of $\mathcal{G}$ induced by $A$, that is,
  $B_A = B \cap (A \times A)$.  
  Since $\mathcal{G}$ is connected, we may assume that the vertex set
  $V=\{1,\dots,p\}$ has been relabeled such that the induced subgraph
  $\mathcal{G}_{\{i+1,\dots,p\}}$ is connected for all $i=1,\dots,p-1$
  \cite[Prop.~1.4.1]{MR2744811}.  Figure~\ref{fig:bidi-relabeling}
  shows an example of a bidirected graph that is labeled in this way.

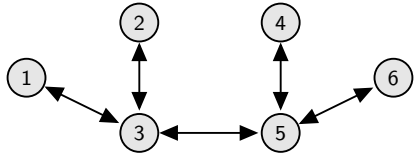
\begin{figure}[t]
  \centering
   \begin{center}
    \scalebox{.9}{
      \begin{tikzpicture}[->,>=triangle 45,shorten >=1pt,
        auto,thick,
        main node/.style={circle,inner
          sep=3pt,fill=gray!20,draw,font=\sffamily}]  
        
        \node[main node] (1) {1};
        \node[main node] (2) [above right=.4cm and 1.25cm of 1]{2}; 
        \node[main node] (3) [below right=.4cm and 1.25cm of 1] {3}; 
        \node[main node] (4) [right=1.5cm of 2] {4};
        \node[main node] (5) [right=1.5cm of 3] {5};
        \node[main node] (6) [above right=.4cm and 1.25cm of 5] {6};
      
        \path[<->,every node/.style={font=\sffamily\small}] 
        (1) edge node {} (3)
        (2) edge node {} (3)
        (3) edge node {} (5)
        (5) edge node {} (6)
        (4) edge node {} (5)
        ;        
      \end{tikzpicture}
    }
  \end{center}
  \caption{A bidirected graph labeled in such a way that for any
    $j$ the nodes $i\ge j$ induce a connected subgraph.}
  \label{fig:bidi-relabeling}
\end{figure}

  We now show how to construct $\Sigma = (\sigma_{kl})\in\M(\mathcal{G})$ such
  that $(\Sigma q)_j=0$ for all $j<p$.  Since $q\not=0$ and $\Sigma$
  will be positive definite, we then have $(\Sigma q)_p\not=0$.  As
  $\Sigma$ must be symmetric, we only have to specify the entries
  $\bar{\sigma}_{kl}$ with $k \leq l$.  

  We construct $\Sigma$ one row (and by symmetry, column) at a time
  according to the following iterative procedure.  At stage
  $i=1,\dots,p$, the first $i-1$ rows and columns have been specified;
  none when $i=1$.  Let $\Sigma_{[i],[i]}=(\sigma_{kl})_{k,l\le i}$ be
  the $i$-th leading principal submatrix.  We set $\sigma_{ii}$ to be
  the smallest natural number with the property that
  $\det(\Sigma_{[i],[i]})>0$; that such a choice is possible is clear
  from a Laplace expansion of the determinant.  For $i=1$, we get
  $\sigma_{ii}=1$.  Next, as long as $i<p$, we choose $i^*\in \{i+1,\dots,p\}$
  such that $i\bi i^*\in B$, which is possible because
  $\mathcal{G}_{\{i,\dots,p\}}$ is connected.  For all $k\ge i+1$ and
  $k\not=i^*$, we set $\sigma_{ik}=0$ if $i\bi k\notin B$ and
  $\sigma_{ik}=1$ if $i\bi k\in B$.  We then complete the $i$-th row
  and column by setting
  \[
    \sigma_{i{i^*}} = - \sum_{l \in V\backslash \{{i^*}\}}
    \sigma_{il} q_l / q_{i^*}\,;
  \]
  the division by $q_{i^*}$ is well-defined as all entries of $q$ are
  nonzero. 
		
  By construction, the matrix $\Sigma$ is positive definite as all
  leading principal minors are positive.  Moreover, $\Sigma_{ij}=0$
  whenever $i\not=j$ and $i\bi j \notin B$.  It follows that
  $\Sigma \in \M(B)=\M(\mathcal{G})$.  Finally, for all $i\le p-1$,
  \[
    \left(\Sigma q\right)_i
    \,=\, \sigma_{ii^*} q_{i^*} + \sum_{l \neq i^*} \sigma_{i l} q_l\,=\,
      - \left(\sum_{l \neq i^*} \sigma_{il} \frac{q_l}{q_{i^*}}\right) q_{i^*} +
      \sum_{l \neq i^*} \sigma_{kl} q_l = 0.
      \qedhere
      \]
\end{proof}

\begin{proof}[Proof of Theorem~\ref{thm:bidirected}]
By Lemma~\ref{lem:saturated}(a), we have $\obs{\mathcal{G}}\le p$.  Hence, we need to
show that the likelihood function is a.s.~unbounded if $n<p$.

Assume that $n < p$.  By Lemma~\ref{lem:zero_entries}, the kernel of
the sample covariance matrix $S_{0,n}$ a.s.~contains a vector
$q$ with all entries nonzero.    By
Lemma~\ref{lem:constr_barsigma}, we may choose a matrix $\Sigma$ such
that $\Sigma q$ has one nonzero entry.  Without loss of generality, we
assume the vertex set to be labeled such that $\Sigma q = c e_p$,
where $c\in\mathbb{R}\setminus\{0\}$ and
$e_p=(0,\dots,0,1)^T\in\mathbb{R}^p$.  Based on these choices, we will
define a sequence of covariance matrices
$\left\{\Sigma_t\right\}_{t=1}^\infty$ in $\M(\mathcal{G})$, with the
property that
$\lim_{t\rightarrow \infty}\ell(\Sigma_t\,|\,S_{0,n}) = \infty$.  This
then implies that the likelihood function is a.s.~unbounded.

%
For $t\ge 0$, define
\begin{align}
  \Sigma_t \;:=\; \Sigma - \frac{1}{\frac{1}{t} + q^T\Sigma q} \Sigma qq^T\Sigma.
  \label{eq:sigma_t}
\end{align}
Since $\Sigma q = c e_p$, the matrix
$(\Sigma q)(\Sigma q)^T=c^2 e_pe_p^T$ is zero with the exception of
the $(p,p)$ entry that equals $c^2>0$.  Hence, $\Sigma_t$ has
zeros in the same entries as $\Sigma\in\M(\mathcal{G})$ does.  Let $K= (\Sigma)^{-1}$.
By the Woodbury matrix identity \citep{woodbury:1950},
\begin{align*}
	K_t := \left(\Sigma_t\right)^{-1} \;=\; K + t qq^T.
\end{align*}
For all $t\ge 0$, the matrix $K_t$ is positive definite because $K$ is
positive definite and $qq^T$ positive semidefinite.  Thus, $\Sigma_t$
is positive definite for all $t\ge 0$ as well.  We conclude that
$\Sigma_t\in\M(\mathcal{G})$ for all $t\ge 0$.  

Inserting $\Sigma_t$ into the log-likelihood function from
\eqref{eq:likelihood}, we have
\begin{align*}
  \ell(\Sigma_t\,|\,S_{0,n}) 
  &=  \log\det\left(K_t\right) - \trace\left(K_t S_{0,n}\right)\\
  &= \log\det \left(K + t qq^T\right) -
    \trace\left(K S_{0,n}\right) - t \,q^T S_{0,n} q\\
  &= \log\det \left(K + t qq^T\right) -
    \trace\left(K S_{0,n}\right)
\end{align*}
because $q$ is in the kernel of $S_{0,n}$.  By the matrix determinant
lemma,
\[
\det \left(K + t qq^T\right) \;=\; (1+t\, q^T\Sigma q) \det(\Sigma),
\]
which converges to infinity as $t\to\infty$ because $\det(\Sigma)>0$
and $q^T\Sigma q>0$ by positive definiteness of $\Sigma$.
\end{proof}

\section{Lower bound from submodels}
\label{sec:lower-bound}

We return to the case where $\mathcal{G}=(V,D,B)$ is an arbitrary, possibly
cyclic, mixed graph.  The following result uses the characterization
of the maximum likelihood threshold for bidirected graphs to yield a lower
bound on $\obs{\mathcal{G}}$.

\begin{theorem}
  \label{thm:inequalities}
  Let $\mathcal{G} = (V,D,B)$ be a mixed graph, and let $C_1,\dots,C_l$ be the
  vertex sets of the connected components of its bidirected part
  $\mathcal{G}_\bi=(V,\emptyset,B)$. Then
  \[
    \obs{\mathcal{G}} \;\ge\; \max_{j=1,\dots,l} \,|\Pa(C_j)|.
  \]
  Moreover, if $n<\obs{\mathcal{G}}$ then $\hat\ell(\mathcal{G}\,|\, S_{0,n})=\infty$ a.s.
\end{theorem}
\begin{proof}
  For $j=1,\dots, l$, let $B_j = B \cap (C_j \times C_j)$ and
  $D_j = D \cap (\Pa(C_j) \times C_j)$.  In other words, $B_j$ is the
  set of bidirected edges between nodes in $C_j$, while $D_j$ is the
  set of directed edges with head in $C_j$.  The sets $B_j$ and
  $D_j$ partition $B$ and $D$, respectively.  The graphs
  $\mathcal{G}_j = (\Pa(C_j), D_j, B_j)$ thus form a decomposition of $\mathcal{G}$. 
  Because each graph $\mathcal{G}_j$ is a subgraph of $\mathcal{G}$,
  Lemma~\ref{lem:saturated}(c) yields that
  \[
    \obs{\mathcal{G}} \;\ge\; \max_{j=1,\dots,l}\,\obs{\mathcal{G}_j}.
  \]
  
  Next, for each $j$, choose a subgraph $\mathcal{H}_j$ of $\mathcal{G}_j$ by taking the
  bidirected part of $\mathcal{G}_j$ and adding for each node in
  $\Pa(C_j)\setminus C_j$ precisely one of its outgoing directed
  edges.  Then let $\mathcal{H}_j^\bi$ be the bidirected graph obtained by
  converting the directed edges of $\mathcal{H}_j$ into bidirected edges.  An
  example is shown in Figure~\ref{fig:intro_example:G4-H4}.  Since in
  $\mathcal{H}_j$ each node $i\in\Pa(C_j)\setminus C_j$ is the parent of
  precisely one node in $C_j$, it follows from Theorem 5 in
  \cite{drton:2008} that $\mathcal{H}_j$ and $\mathcal{H}_j^\bi$ define the same set of
  covariance matrices.  Consequently,
  \[
    \M(\mathcal{H}_j^\bi) \;=\; \M(\mathcal{H}_j) \;\subseteq \; \M(\mathcal{G}_j).
  \]
  Now use Lemma~\ref{lem:saturated}(c) and apply
  Theorem~\ref{thm:bidirected} to the connected bidirected graph
  $\mathcal{H}_j^\bi$ to conclude that
  \[
    \obs{\mathcal{G}_j}\;\ge\;\obs{\mathcal{H}_j}  \;=\;\obs{\mathcal{H}_j^\bi} \;=\; |\Pa(C_j)|.
    \qedhere
  \]
\end{proof}

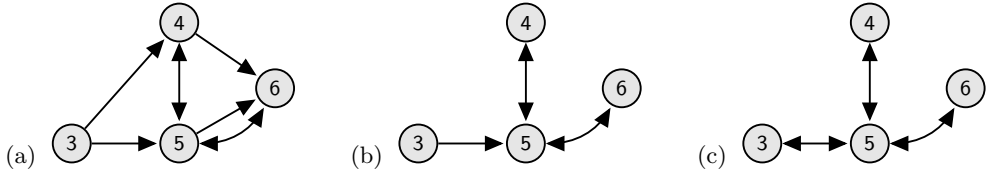
\begin{figure}[t]
  \centering
   \begin{center}
     \begin{tabular}{lll}
       (a)
       \scalebox{.9}{
      \begin{tikzpicture}[->,>=triangle 45,shorten >=1pt,
        auto,thick,
        main node/.style={circle,inner
          sep=3pt,fill=gray!20,draw,font=\sffamily}]  
        
        \node[main node] (3) {3}; 
        \node[main node] (5) [right=1cm of 3] {5};
        \node[main node] (4) [above=1.2cm of 5] {4};
        \node[main node] (6) [above right=.4cm and 1cm of 5] {6};
      
        \path[every node/.style={font=\sffamily\small}] 
        (3) edge node {} (4)
        (3) edge node {} (5)
        (4) edge node {} (6)
        (5) edge node {} (6);
        \path[<->,every node/.style={font=\sffamily\small}, bend right] 
        (5) edge node {} (6);
        \path[<->,every node/.style={font=\sffamily\small}] 
        (4) edge node {} (5)
        ;        
      \end{tikzpicture}
    } \ \  &
        (b)
       \scalebox{.9}{
      \begin{tikzpicture}[->,>=triangle 45,shorten >=1pt,
        auto,thick,
        main node/.style={circle,inner
          sep=3pt,fill=gray!20,draw,font=\sffamily}]  
        
        \node[main node] (3) {3}; 
        \node[main node] (5) [right=1cm of 3] {5};
        \node[main node] (4) [above=1.2cm of 5] {4};
        \node[main node] (6) [above right=.4cm and 1cm of 5] {6};
      
        \path[every node/.style={font=\sffamily\small}] 
        (3) edge node {} (5);
        \path[<->,every node/.style={font=\sffamily\small}, bend right] 
        (5) edge node {} (6);
        \path[<->,every node/.style={font=\sffamily\small}] 
        (4) edge node {} (5)
        ;        
      \end{tikzpicture}
    } \ \   &
        (c)
       \scalebox{.9}{
      \begin{tikzpicture}[->,>=triangle 45,shorten >=1pt,
        auto,thick,
        main node/.style={circle,inner
          sep=3pt,fill=gray!20,draw,font=\sffamily}]  
        
        \node[main node] (3) {3}; 
        \node[main node] (5) [right=1cm of 3] {5};
        \node[main node] (4) [above=1.2cm of 5] {4};
        \node[main node] (6) [above right=.4cm and 1cm of 5] {6};
      
        \path[<->,every node/.style={font=\sffamily\small}] 
        (3) edge node {} (5);
        \path[<->,every node/.style={font=\sffamily\small}, bend right] 
        (5) edge node {} (6);
        \path[<->,every node/.style={font=\sffamily\small}] 
        (4) edge node {} (5)
        ;        
      \end{tikzpicture}
    }   
      \end{tabular}
  \end{center}
  \caption{In reference to the graph from
    Figure~\ref{fig:intro_example}(b), the panels show: (a) the
    connected component
    $\mathcal{G}_4$ with vertex set $C_4=\{4,5,6\}$, (b) a
    choice for $\mathcal{H}_4\subset \mathcal{G}_4$, and (c) the bidirected graph $\mathcal{H}_4^\bi$.}
  \label{fig:intro_example:G4-H4}
\end{figure}

\section{Numerical experiment}
\label{sec:numerical-experiment}

A model with low maximum likelihood threshold is amenable to standard
likelihood inference even when the modeled observations are
high-dimensional and the sample size is rather small. We demonstrate
this for a structural equation model associated with a directed graph
and allowing for cycles.
Specifically, we consider a graph $\mathcal{G}_p=(V_p,E_p)$ with
an even number $p$ of nodes.  As previously, we enumerate the vertex
set as $V_p=\{1,\dots,p\}$.  Let $p'=p/2$, and define the edge set as
$E_p=E^{(1)}_p\cup E^{(2)}_p\cup E^{(3)}_p$, where
\begin{align*}
  E^{(1)}_p &= \left\{ i \to i+1: i=1,\dots,p'-1\right\}\cup\{ p'\to 1\},\\
  E^{(2)}_p &= \left\{ i+3\to i : i=1,\dots,p'-3\right\}\cup\{ 1\to p'-2\}\cup\{2 \to p'-1\}\cup\{3 \to p'\},\\
  E^{(3)}_p &= \left\{ p'+i\to i : i=1,\dots,p'\right\}.
\end{align*}
The first set of edges defines a directed cycle of length $p'$, and
the second set of edges gives many shorter cycles of length 4.  The
third set of edges attaches, in bipartite fashion, additional nodes
that play the role of covariates; one covariate for each node in the
long cycle.  Figure~\ref{fig:numerical-exp} illustrates this with a
picture of $\mathcal{G}_{40}$.

As a statistical problem we consider testing absence of the edge
$1\to 2$ from the graph $\mathcal{G}_{100}$.  In other words, we test the
hypothesis $H_0:\lambda_{12}=0$ in the model given by $\mathcal{G}_{100}$.  The
parametrization for $\mathcal{G}_{100}$ is generically one-to-one as can be
confirmed, for instance, using the half-trek criterion
\citep{foygel:draisma:drton:2012,semid}.  Assuming zero means for the
$p=100$ dimensional observation vector, the model corresponds to a
$p+3p/2=250$ dimensional set of covariance matrices.  We test $H_0$ using the likelihood ratio test for three rather
small sample sizes, namely, $n=15$, $20$, and $25$.  Our main
result guarantees that the test is well-defined as the log-likelihood
function for $\mathcal{G}_{100}$ a.s.~admits a finite supremum at
these sample sizes.  The optimization needed to compute the likelihood
ratio statistics is performed using the algorithm of \cite{drton-fox-wang}.

\begin{figure}
  \centering
  \scalebox{.9}{
    \begin{tikzpicture}[->,>=latex',shorten >=1pt,
      auto,thick,
      main node/.style={circle,inner
        sep=3pt,fill=gray!20,draw,font=\sffamily}]
      
      \foreach \a in {1,2,...,20}{
        \draw     node[main node]  (in\a) at  (\a*360/20: 3cm) { }; }
      
      \foreach \a in {1,2,...,20}{
        \draw     node[main node]  (out\a) at  (\a*360/20: 4cm) {
        }; 
        \path[->,every node/.style={font=\sffamily\small}]
        (out\a) edge node {} (in\a);
      }

      \path[->,every node/.style={font=\sffamily\small},bend right]
      (in1) edge (in18)
      (in2) edge (in19)
      (in3) edge (in20)
      (in4) edge (in1)
      (in5) edge (in2)
      (in6) edge (in3)
      (in7) edge (in4)
      (in8) edge (in5)
      (in9) edge (in6)
      (in10) edge (in7)
      (in11) edge (in8)
      (in12) edge (in9)
      (in13) edge (in10)
      (in14) edge (in11)
      (in15) edge (in12)
      (in16) edge (in13)
      (in17) edge (in14)
      (in18) edge (in15)
      (in19) edge (in16)
      (in20) edge (in17)
      ;

      \path[->,every node/.style={font=\sffamily\small}] 
      (in1) edge (in2);
      \path[->,every node/.style={font=\sffamily\small}]
      (in2) edge (in3)
      (in3) edge (in4)
      (in4) edge (in5)
      (in5) edge (in6)
      (in6) edge (in7)
      (in7) edge (in8)
      (in8) edge (in9)
      (in9) edge (in10)
      (in10) edge (in11)
      (in11) edge (in12)
      (in12) edge (in13)
      (in13) edge (in14)
      (in14) edge (in15)
      (in15) edge (in16)
      (in16) edge (in17)
      (in17) edge (in18)
      (in18) edge (in19)
      (in19) edge (in20)
      (in20) edge (in1)
      ;



      \end{tikzpicture}
    }   
  \caption{A directed graph with cycles and maximum in-degree 3.}
  \label{fig:numerical-exp}
\end{figure}
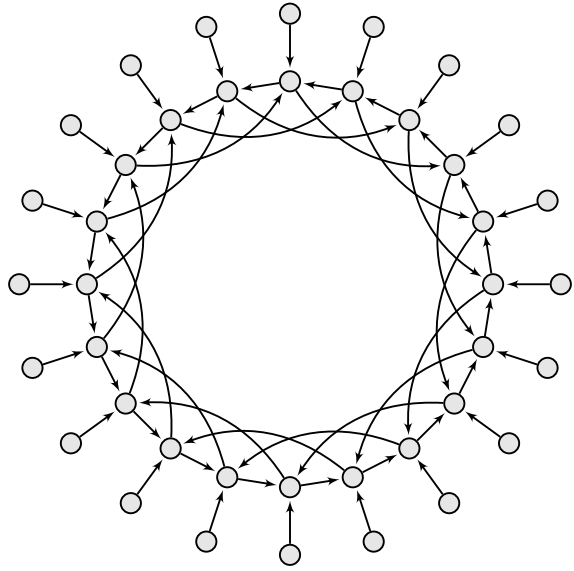

For each sample size, we use 200 Monte Carlo simulations to
approximate the size of the test as well as its power at nonzero
values of $\lambda_{12}$.  Specifically, we consider the setting where
$\lambda_{12}$ ranges through $[-1,1]$, and all other edge
coefficients are set to $1/3$.  We consider nominal significance level
0.05 and calibrate the likelihood ratio test using a chi-square
distribution with 1 degree of freedom.  A chi-square limiting
distribution cannot always be expected \citep{drton:lrt}, but is valid
at the considered identifiable parameter.  The power functions
are plotted in Figure~\ref{fig:power}.  The asymptotically calibrated
test clearly exhibits good power at stronger signals and is seen to be
only slightly liberal.  This suggests that likelihood inference allows
one to perform model selection in high-dimensional but sparse cyclic
models.

\begin{figure}[t]
  \centering
  \includegraphics[scale=.75]{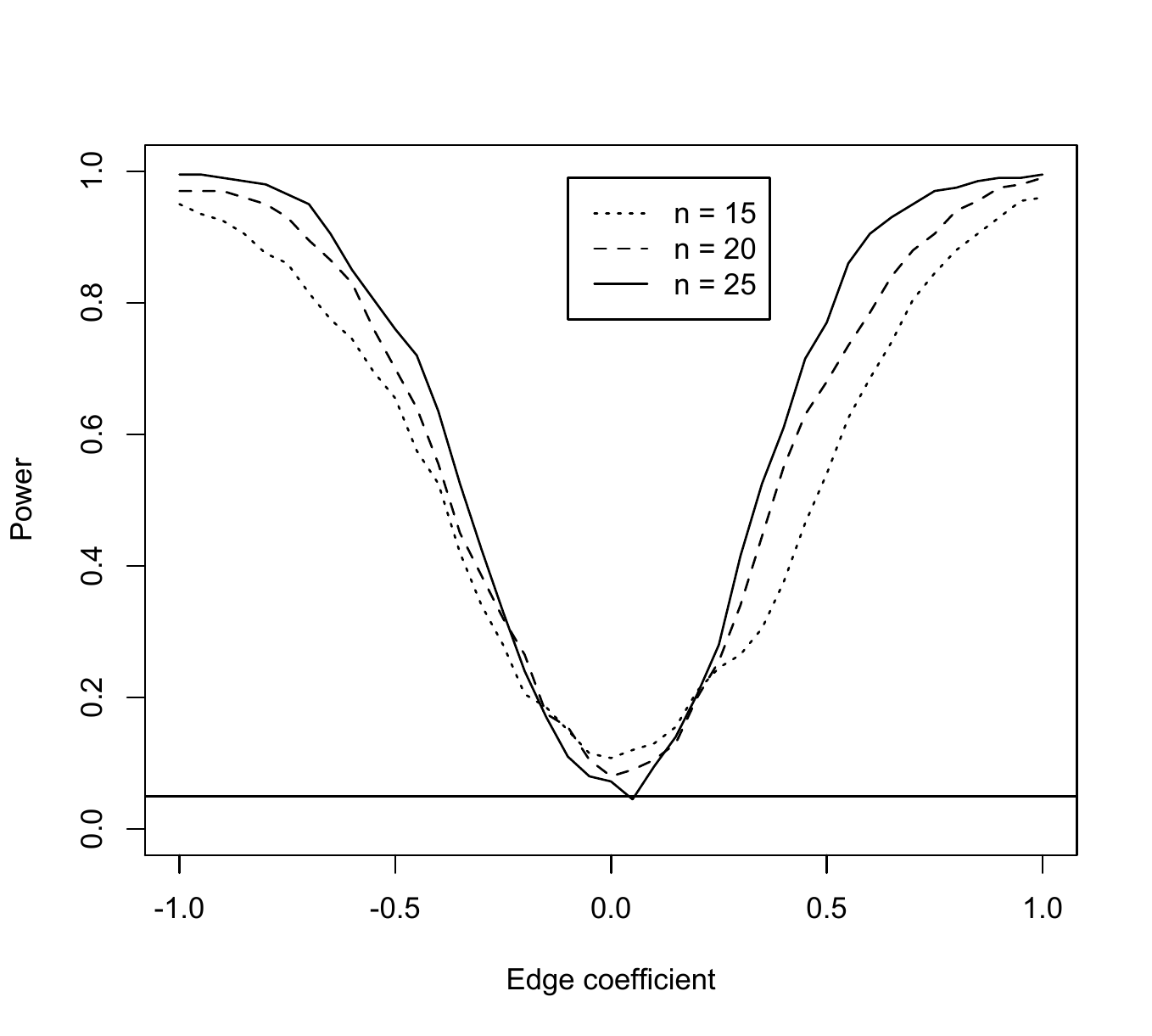}
  \caption[power]{Monte Carlo approximation to power function of a
    likelihood ratio test.}
  \label{fig:power}
\end{figure}

\section{Connections to undirected graphical models}
\label{sec:undirected}

The structural equation models we considered are closely related to
the Gaussian graphical models given by undirected graphs
\citep{lauritzen:1996}.  The latter are dual to the models given by
bidirected graphs in the sense that it is not the covariance matrix
but its inverse that is supported over the graph \citep{MR1380485}.  To be more precise,
let $\bar{\mathcal{G}}=(V,E)$ be an undirected graph, whose edges we take to be
unordered pairs $\{i,j\}$ comprised of two distinct nodes $i,j\in
V=\{1,\dots,p\}$.  
Let
$\M(\bar{\mathcal{G}})$ be the cone of positive definite $p \times p$ matrices
$K=(\kappa_{ij})$ with $\kappa_{ij} = 0$ whenever $i \neq j$ and
$\{i,j\} \notin E$.  The Gaussian graphical model given by $\bar{\mathcal{G}}$ is
the set of multivariate normal distributions $\mathcal{N}(\mu,\Sigma)$
with arbitrary mean vector $\mu\in\mathbb{R}^p$ and a covariance
matrix that is constrained to have $\Sigma^{-1}\in\M(\bar{\mathcal{G}})$.


Suppose $\mathcal{G}=(V,D,\emptyset)$ is an acyclic digraph that is perfect,
that is, $i,j\in\pa(k)$ implies that $i$ and $j$ are adjacent.  Then
$\mathit{PD}(\mathcal{G})$ is equal to the set of covariance matrices of the
Gaussian graphical model given by the skeleton of $\mathcal{G}$ \cite[see
e.g.][Corollary~4.1, 4.3]{MR1439312}.  The skeleton is the undirected
graph $\mathcal{G}_-=(V,E)$ with $\{i,j\}\in E$ whenever $i$ and $j$ are
adjacent in $D$.  When $\mathcal{G}$ is perfect then $\mathcal{G}_-$ is chordal.
Theorem~\ref{thm:dags} implies that the maximum likelihood threshold
of the Gaussian graphical model of a chordal graph is the maximum
clique size; see \citet[Theorem 7]{MR739282} or \citet[Theorem
3.2]{buhl:1993}.

The maximum likelihood threshold of graphical models given by
non-chordal graphs is more subtle to derive.  Many interesting results
exist but the threshold has not yet been determined in generality
\citep{buhl:1993,MR3014306,Sullivant:Gross}.  Moreover, it has been
shown for a sample size below the threshold that the likelihood may be
bounded with positive probability.  In the remainder of this section,
we focus on chordless cycles, which were the first known examples of
this phenomenon.  We note that in the literature the maximum
likelihood threshold for Gaussian graphical models is typically
introduced as the minimum sample size at which the likelihood function
admits a maximizer.  The maximizer is then unique by strict convexity
of the log-likelihood function as a function of the inverse covariance
matrix.  By the duality theory in \cite{MR2440363}, if the likelihood
function of a Gaussian graphical model does not achieve its maximum,
then it is unbounded; see also Theorem 9.5 in \cite{MR489333}.

\begin{example}
  Let $\bar{\mathcal{C}}_p$ be the undirected chordless cycle with vertex set
  $V=\{1,\dots,p\}$ and edge set
  $E=\{\{1,2\},\{2,3\},\dots,\{p-1,p\},\{1,p\}\}$ for $p\ge 3$.
  Assuming the mean vector $\mu$ to be zero, the Gaussian graphical
  model given by $\bar{\mathcal{C}}_p$ has maximum likelihood threshold
  $\obs{\bar{\mathcal{C}}_p}=3$.  However, if $n=2$, then the likelihood function
  of the model with zero means is bounded, and achieves its maximum,
  with positive probability \citep[Theorem 4.1]{buhl:1993}.
\end{example}

Let $\M(\bar{\mathcal{C}}_p)^{-1}$ be the set of matrices with an inverse in
$\M(\bar{\mathcal{C}}_p)$.  In other words, $\M(\bar{\mathcal{C}}_p)^{-1}$ is the set of
covariance matrices of the graphical model for $\bar{\mathcal{C}}_p$.  If an
acyclic digraph $\mathcal{G}=(V,D,\emptyset)$ satisfies
$\M(\mathcal{G})\subseteq \M(\bar{\mathcal{C}}_p)^{-1}$ then $\M(\mathcal{G})$ has smaller dimension
than $\M(\bar{\mathcal{C}}_p)^{-1}$.  If $\M(\mathcal{G})\supseteq \M(\bar{\mathcal{C}}_p)^{-1}$, then
the dimension of $\M(\mathcal{G})$ is larger.  However, a subset of the same
dimension is found when considering digraphs with cycles.
Specifically, take $\mathcal{C}_p=(V,D,\emptyset)$ to be the digraph with vertex
set $V=\{1,\dots,p\}$ and edge set
$D=\{1\to 2,2\to 3,\dots,p-1\to p,p\to 1\}$.  Then
$\M(\mathcal{C}_p)\subseteq \M(\bar{\mathcal{C}}_p)^{-1}$.  Indeed, if
$\Sigma=(I-\Lambda)^{-T}\Omega(I-\Lambda)^{-1}$ for $\Lambda\in\RDreg$
and $\Omega\in\M(B)$, then
$\Sigma^{-1}=(I-\Lambda)\Omega^{-1}(I-\Lambda)^T$ has entries
\begin{equation}
  \label{eq:p-cycle-matrix}
\Sigma^{-1}_{ij}=
\begin{cases}
  \frac{1}{\omega_{ii}} + \frac{\lambda_{i,i+1}^2}{\omega_{i+1,i+1}} &\text{ if
  } i=j,\\
  -\frac{\lambda_{i,i+1}}{\omega_{i+1,i+1}} &\text{ if }
  j\in\{i-1,i+1\},\\
  0 &\text{ otherwise}.
\end{cases}
\end{equation}
Here, we identify $p+1\equiv 1$.
Recall that for a digraph $\Omega=(\omega_{ij})$ is diagonal.
The zeros in~(\ref{eq:p-cycle-matrix}) now confirm that
$\M(\mathcal{C}_p)\subseteq \M(\bar{\mathcal{C}}_p)^{-1}$.  Moreover, $\M(\mathcal{C}_p)$ is a
full-dimensional subset as both $\M(\mathcal{C}_p)$ and $\M(\bar{\mathcal{C}}_p)$ clearly have
dimension $2p$.

\begin{example}
  \label{ex:4cycle}
  The graphs $\bar{\mathcal{C}}_4$ and $\mathcal{C}_4$ are depicted in
  Figure~\ref{fig:4cycles}.  A matrix in $\M(\mathcal{C}_4)$ is parameterized as
  \begin{equation}
    \label{eq:4cycle-matrix}
    \begin{pmatrix}
      \frac{1}{\omega_{11}} + \frac{\lambda_{12}^2}{\omega_{22}} 
      & -\frac{\lambda_{12}}{\omega_{22}} & 0 &
      -\frac{\lambda_{41}}{\omega_{11}}\\
      -\frac{\lambda_{12}}{\omega_{22}} & \frac{1}{\omega_{22}} +
      \frac{\lambda_{23}^2}{\omega_{33}}  & -\frac{\lambda_{23}}{\omega_{33}}
      & 0\\
      0 & -\frac{\lambda_{23}}{\omega_{33}} & \frac{1}{\omega_{33}} +
      \frac{\lambda_{34}^2}{\omega_{44}}  &
      -\frac{\lambda_{34}}{\omega_{44}}\\
      -\frac{\lambda_{41}}{\omega_{11}}& 0 &-\frac{\lambda_{34}}{\omega_{44}}
      &\frac{1}{\omega_{44}} + 
      \frac{\lambda_{41}^2}{\omega_{11}}
    \end{pmatrix}.
  \end{equation}
\end{example}

\begin{figure}[t]
  \centering
   \begin{center}
     \begin{tabular}{lll}
       (a)
       \scalebox{.9}{
      \begin{tikzpicture}[shorten >=1pt,
        auto,thick,
        main node/.style={circle,inner
          sep=3pt,fill=gray!20,draw,font=\sffamily}]  
        
        \node[main node] (1) {1}; 
        \node[main node] (2) [right=1.5cm of 1] {2};
        \node[main node] (3) [below=1cm of 2] {3};
        \node[main node] (4) [left=1.5cm of 3] {4};
      
        \path[every node/.style={font=\sffamily\small}] 
        (1) edge (2) (2) edge (3) (3) edge (4) (4) edge (1);
        ;        
      \end{tikzpicture}
    } \qquad\qquad\qquad  &
        (b)
       \scalebox{.9}{
      \begin{tikzpicture}[->,>=triangle 45,shorten >=1pt,
        auto,thick,
        main node/.style={circle,inner
          sep=3pt,fill=gray!20,draw,font=\sffamily}]  
        
        \node[main node] (1) {1}; 
        \node[main node] (2) [right=1.5cm of 1] {2};
        \node[main node] (3) [below=1cm of 2] {3};
        \node[main node] (4) [left=1.5cm of 3] {4};
      
        \path[every node/.style={font=\sffamily\small}] 
        (1) edge (2) (2) edge (3) (3) edge (4) (4) edge (1);
        ;        
      \end{tikzpicture}
    }   
      \end{tabular}
  \end{center}
  \caption{(a) The undirected cycle $\bar{\mathcal{C}}_4$.  (b) The
    directed cycle $\mathcal{C}_4$.}
  \label{fig:4cycles}
\end{figure}
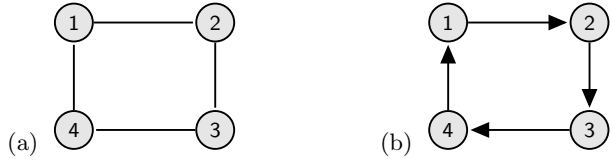

By Theorem 4.1 in \cite{buhl:1993}, $\obs{\bar{\mathcal{C}}_p}=3$ for all
$p\ge 3$.  In contrast, our new Theorem~\ref{thm:intro} implies that
$\obs{\mathcal{C}_p}=2$ for all $p\ge 3$.  Consequently, it must hold that
$\M(\mathcal{C}_p)\subsetneq \M(\bar{\mathcal{C}}_p)^{-1}$.  Indeed, the set $\M(\mathcal{C}_p)$
comprises matrices that satisfy an additional inequality.  Applying a
trick used by \cite{drtonyu:2010} in a different context, observe that
negating the entry $\lambda_{12}$ changes only the entries
$\Sigma^{-1}_{12}$ and $\Sigma^{-1}_{21}$, which are negated.  All
other entries of $\Sigma^{-1}$ are preserved under the sign change of
$\lambda_{12}$.  The inequality is obtained by noting that not every
positive definite matrix in $\M(\mathcal{C}_p)$ remains positive definite after
negation of a single off-diagonal entry \cite[Example
5.2]{drtonyu:2010}.  We conclude that if a sample of size 2 has the
likelihood function given by $\bar{\mathcal{C}}_p$ unbounded, then the divergence
occurs only along sequences of matrices that do not represent a system
with a feedback cycle as in $\mathcal{C}_p$.




\section{Discussion}\label{sec:discussion}


Our main result, Theorem~\ref{thm:intro}, determines the maximum
likelihood threshold of any linear structural equation model.  This
threshold is the smallest integer $N$ such that the Gaussian
likelihood function is a.s.~bounded for all samples of size at least
$N$.  According to our result, the maximum likelihood threshold of
models with feedback loops is surprisingly low.  Indeed, the maximum
likelihood threshold of any digraph, acyclic or not, is equal to the
maximum in-degree plus one.  In contrast, bidirected edges, which
represent the effects of unmeasured confounders, can result in a large
maximum likelihood threshold by merely forming long paths.  If
$\mathcal{G}$ is a bidirected spanning tree, then there are only $p-1$
edges yet $\obs{\mathcal{G}}=p$, which is the largest possible value
by Lemma~\ref{lem:saturated}(a).

When the structural equation model is given by an acyclic digraph,
boundedness of the likelihood function implies that the maximum is
achieved.  As we emphasized in Remark~\ref{rem:not-closed}, the
question of when the maximum is a.s.~achieved is still poorly
understood for general mixed graphs and constitutes an important topic
for future work.
	
\section*{Acknowledgments}

We would like to thank Caroline Uhler for helpful discussions.  This material is based upon work supported by the U.S. National Science Foundation under Grant No.~DMS 1712535.

\bibliographystyle{biometrika}
\small
\bibliography{MLE_Existence}	
	
\end{document}